\documentclass{article}

\usepackage{xspace}
\usepackage{fancyvrb}
\usepackage{hyperref}
\usepackage[T1]{fontenc}
\usepackage[utf8]{inputenc}
\usepackage{amsthm}
\usepackage{amssymb}
\usepackage{enumerate} 
\usepackage{tikz-cd}
\usepackage{mathrsfs}
\usepackage{mathtools}
\usepackage[nameinlink]{cleveref}
\usepackage{dsfont}

\newtheorem{theorem}{Theorem}[section]

\newtheorem{lemma}{Lemma}[subsection]
\newtheorem{prop}{Proposition}[subsection]
\newtheorem{rem}{Remark}

\newcommand{\C}{\mathbb{C}}
\newcommand{\Ff}{\mathrm{F}}
\newcommand{\Ee}{\mathrm{E}}
\newcommand{\gl}{\mathrm{GL}}
\newcommand{\Mp}{\mathrm{Mp}}
\newcommand{\ra}{\rightarrow}
\newcommand{\sra}{\twoheadrightarrow}
\newcommand{\hra}{\hookrightarrow}
\newcommand{\Rep}{\mathrm{Rep}}
\newcommand{\id}{\mathrm{Ind}}

\newcommand{\abs}{\lvert-\lvert}
\newcommand{\NN}{\mathbb{N}}
\renewcommand*{\det}{\qopname\relax o{det}}

\newcommand{\mvw}{^{\mathrm{MVW}}}
\newcommand{\irr}{\mathfrak{Irr}}
\newcommand{\ain}[3]{#1\in\{#2,\ldots,#3\}}
\newcommand{\ho}{\mathrm{Hom}}
\newcommand{\soc}{\mathrm{soc}}
\newcommand{\cosoc}{\mathrm{cosoc}}
\newcommand\restr[2]{{\left.\kern-\nulldelimiterspace #1\vphantom{\big|}\right|_{#2}}}

\begin{document}

\title{A note on the Howe Duality conjecture for symplectic-orthogonal and unitary pairs}
\author{Johannes Droschl}
\maketitle
\begin{abstract}
    In this short note we expand on recent results on the degenerate principle series $I(s,\chi)$ of classical groups associated to $s\in \C$ and a quadratic character $\chi$. In particular, we strengthen the result for $s\in \mathbb{R}_{\ge 0}$, which allows us to give as a corollary a new proof of the Howe duality conjecture for symplectic-orthogonal and unitary pairs.
\end{abstract}
\emph{Contact: }\url{johannes.carl.droschl@univie.ac.at}
\section{Introduction}
Let $\Ff$ be a local, non-archimedean field of characteristic different from $2$. Let $\epsilon\in\{\pm 1\}$, $\Ee$ be either $\Ff$ or a quadratic extension of $\Ff$ and $W$, respectively $V$, a $-\epsilon$-, respectively $\epsilon$-, hermitian vector-space of dimension $n$, respectively $m$. If $\Ee=\Ff$ and $W$ is symplectic, we moreover assume $n$ to be even. We denote the symmetry groups of $W$ and $V$ by $G(W)$ and $H(V)$. To be more precise, if $\Ee=\Ff$, we consider in the case of a symplectic space the corresponding symplectic group and in the case of a quadratic space the corresponding full orthogonal group. If $\Ee\neq \Ff$ we consider the corresponding unitary group. To incorporate the case of $W$ being symplectic and $V$ having odd dimension, one is forced to consider the metaplectic group associated to $W$ instead of its symplectic group. Due to some technical details, which we will describe towards the end of this introduction, we will however not deal with this case.

Then $G(W)\times H(V)$ forms a dual-pair in the metaplectic group $\Mp(V\otimes W)$. For a fixed smooth, additive character $\psi\colon \Ff\ra \C^\times$, we let $\omega_{V,W,\psi}$ be the Weil-representation of the group $\Mp(V\otimes W)$.
For $\pi$ an irreducible, smooth representation of $G(W)$, we define as usual the smooth $H(W)$-representation $\Theta_{V,W,\psi}(\pi)$ as the maximal $\pi$-isotypical quotient of $\restr{\omega_{V,W,\psi}}{G(W)\times H(V)}$.
In \cite{Kudla1986} Kudla showed that if $\pi$ is irreducible, then $\Theta_{V,W,\psi}(\pi)$ is of finite length. One can thus consider the cosocle of $\Theta_{V,W,\psi}(\pi)$, \emph{i.e.} its largest semi-simple quotient, which is usually denoted by $\theta_{V,W,\psi}(\pi)$.

The study of the theta lifts $\Theta_{V,W,\psi}(\pi)$ and $\theta_{V,W,\psi}(\pi)$ lies at the heart of the local theta correspondence and their behavior is to a large extend governed by what is now known as \emph{Howe duality} (in type I). It was conjectured in \cite{Howe1979seriesAI}, \cite{Howe2} and proven in \cite{Waldspurger} in the case $p\neq 2$. In \cite{GanTakeda} a new proof without the assumption on $p$ was given and in \cite{Gan2017} the remaining case of quaternionic dual pairs was covered.
\begin{theorem}[{Howe duality conjecture in type I, {\cite{Waldspurger}, \cite{GanTakeda}, \cite{Gan2017}}}]
    Let $\pi$ be an irreducible, smooth representation of $G(W)$.
    \begin{enumerate}
        \item If $\Theta_{V,W,\psi}(\pi)$ is non-zero, then $\theta_{V,W,\psi}(\pi)$ is irreducible.
        \item If $\pi'$ is a second irreducible, smooth representation of $G(W)$ with \[\theta_{V,W,\psi}(\pi)\cong \theta_{V,W,\psi}(\pi')\neq 0,\] then $\pi\cong \pi'$.
    \end{enumerate}
\end{theorem}
The non-vanishing of $\Theta_{V,W,\psi}(\pi)$ can be tightly controlled by the so-called \emph{conservation relation}, which was proved in \cite{sun}. Another proof of it was given in \cite{DroKudRa} by studying the degenerate principal series $I(s,\chi)$ associated to $s\in \C$ and a quadratic character $\chi$, a smooth representation of $G(W)\times G(W)$.
The main result of this paper is a slight strengthening of \cite[Theorem 3]{DroKudRa} for $s\in \mathbb{R}_{\ge 0}$.
\begin{theorem}\label{T:intro}
    Let $\pi,\pi'$ be two irreducible, smooth representation of $G(W)$ and let $s\in \mathbb{R}_{\ge 0}$. Then
    \[\dim_\C\ho_{G(W)\times G(W)}(I(s,\chi),\pi\otimes \chi(\det_{G(W)})\pi')=\delta_{\pi',\pi^\lor}.\]
\end{theorem}
Similarly to \cite{GanTakeda}, a key tool in our proof are \emph{derivatives}, see \cite{AtoMin}.
As an easy corollary of the above theorem we are then able to deduce the Howe duality conjecture for the aforementioned dual-pairs in type I.

Finally, let us remark two things. Firstly, note that \Cref{T:intro} cannot be extended to arbitrary $s\in \C$, see for example \Cref{Rem}.
Secondly, one would obtain a proof for all dual-pairs of type I, if \cite[Theorem 3]{DroKudRa} could be shown also for these cases. In the metaplectic-case, the main obstruction lies at the moment in the fact that the proof of \cite[Theorem 3]{DroKudRa} hinges on Howe-duality in type II, which at the moment is only known for $\mathrm{GL}_n(\Ff)$, \emph{cf.} \cite{Minguez}, and not for its metaplectic-twofold cover, which would be needed for the metaplectic case.
For quaternionic dual pairs, the story is more involved since we have no longer the MVW-involution at our disposal, which plays a critical role in the proof of \cite[Theorem 3]{DroKudRa}.
\subsection*{Acknowledgements}
I want to express my gratitude towards Wee Teck Gan for pointing out that the results of the authors PhD thesis could be used to obtain a new proof of the Howe duality conjecture, as well as Binyong Sun, for encouraging me during my PhD-defense to consider this application. Finally I want to thank Alberto M{\'i}nguez for taking the time to listen to earlier versions of this note.
This work has been supported by the projects PAT4628923 and PAT4832423 of the Austrian Science Fond (FWF).
\section{Preliminaries}
\counterwithin{theorem}{section}
Let $V$ and $W$ be as in the introduction and denote by $\abs\colon\Ee\ra \mathrm{R}_{\ge 0}$ the absolute value of $\Ee$.
Let $G$ be a reductive $\Ff$-group and by abuse of notation we will also write $G$ for the $\Ff$-points of $G$. Throughout the rest of the paper we will denote by $\Rep(G)$ the category of smooth, complex, finite length representations of $G(\Ff)$. The trivial representation is denoted by $\mathbf{1}_G$, for $\pi,\pi'\in \Rep(G)$ we mark surjective and injective morphisms from $\pi$ to $\pi'$ by $\pi\sra\pi'$ and $\pi\hra\pi'$, and the identity morphism on $\pi$ is denoted by $\mathds{1}_\pi$. We denote the contragradient representation of $\pi$ by $\pi^\lor$ and its socle respectively cosocle, \emph{i.e.} the maximal semi-simple subrepresentation respectively quotient, by $\soc(\pi)$ respectively $\cosoc(\pi)$. We let $K(G)$ be the Grothendieck group of $\Rep(G)$ and denote by $[\pi]\in K(G)$ the element corresponding to $\pi\in \Rep(G)$. Denote by $S(G)$ the space of locally constant and compactly supported functions $f\colon G\ra \C$ on which $G\times G$ acts by left-right translation. Finally, write for the set of isomorphism classes of irreducible representations in $\Rep(G)$ the symbol $\irr(G)$. 

For $P\subseteq G$ a parabolic subgroup with Levi-decompostion $P=MN$, we denote normalized parabolic induction and the Jacquet functor by
\[\id_P^G\colon \Rep(M)\ra\Rep(G),\, r_P\colon \Rep(G)\ra\Rep(M).\]
These functors are exact and by \emph{Frobenius reciprocity} $r_P$ is the left adjoint of $\id_P^G$. If we denote by $\overline{P}$ the opposite parabolic subgroup, $r_{\overline{P}}$ is the right-adjoint of $\id_P^G$, a fact which we will refer to as \emph{Bernstein reciprocity}.
We say $\rho\in\irr(G)$ is cuspidal if for all non-trivial parabolic subgroups $P\subseteq G$, $r_P(\rho)=0$.
Finally, $\id_P^G$ commutes with $(-)^\lor$, whereas $(-)^\lor\circ r_P=r_{\overline{P}}\circ (-)^\lor$. 
\subsection{ The general linear group}
We write $G_n\coloneq \mathrm{GL}_n(\Ee)$ and fix the minimal parabolic subgroup $B_n$ of upper triangular matrices. We are going to recall the representation theory of the general linear group as presented in {\cite{Zel}}.
To a partition $(\alpha_1,\ldots,\alpha_k)$ of $n$ we associate the parabolic subgroup of $G_n$ containing $B_n$, whose Levi-component consists of the block-diagonal matrices of the form $G_{\alpha_1}\times\ldots\times G_{\alpha_k}$. Parabolic induction and the Jacquet functor with respect to $P_\alpha$ are then denoted as usual by
\[\pi_1\times\ldots\times\pi_k\coloneq \id_{P_\alpha}^{G_n}(\pi_1\otimes\ldots\otimes \pi_k)\text{ and } r_\alpha\coloneq r_{P_\alpha}.\]
The product \[\times\colon \Rep(G_n)\times \Rep(G_m)\ra \Rep(G_{n+m})\] is not commutative, however since parabolic induction is exact it induces a product on \[K(G_n)\times K(G_m)\ra K(G_{n+m}),\] which turns out to be commutative.
In particular if $\pi\times\pi'$ is irreducible, $\pi\times\pi'\cong \pi'\times \pi$. For $r\in \NN$, we write \[\pi^{\times r}\coloneq \overbrace{\pi\times\ldots\times\pi}^{ r}.\]
Finally, we recall the cuspidal support of $\pi$, \emph{i.e.} the well-defined, formal sum $\rho_1+\ldots+\rho_k$, where $\rho_1\otimes\ldots\otimes \rho_k$ is a cuspidal representation such that $\pi\hra\rho_1\times\ldots\times\rho_k$.

We recall the following properties of induced representations.
\begin{theorem}[{\cite{Zel}}]\label{T:induced}
Let $\xi\colon \Ee^\times\ra\C^\times$ be a character, $r\in \mathbb{N}$ and $s\in \C$. 
\begin{enumerate}
    \item The representation $\xi^{\times r}$ is irreducible.
    \item The representation $\lvert\det_{G_n}\lvert^s\times\xi^{\times r}$ is irreducible if and only if \[\xi\notin \{\lvert-\lvert^{s-\frac{n+1}{2}},\lvert-\lvert^{s+\frac{n+1}{2}}\}.\]
    \item If $\xi=\lvert-\lvert^{s-\frac{n+1}{2}}$, \[\cosoc(\lvert\det_{G_n}\lvert^s\times\xi)=\lvert\det_{G_{n+1}}\lvert^{s-\frac{1}{2}}=\soc(\xi\times \lvert\det_{G_n}\lvert^s).\]
    \item If $\xi=\lvert-\lvert^{s+\frac{n+1}{2}}$,
\[\soc(\lvert\det_{G_n}\lvert^s\times\xi)=\lvert\det_{G_{n+1}}\lvert^{s+\frac{1}{2}}=\cosoc(\xi\times \lvert\det_{G_n}\lvert^s).\]
\item For a partition $\alpha=(\alpha_1,\ldots,\alpha_k)$, 
\[r_\alpha(\lvert\det_{G_n}\lvert^s)=\lvert\det_{G_{\alpha_1}}\lvert^{-\frac{n-\alpha_1}{2}}\otimes\ldots\otimes \lvert\det_{G_{\alpha_i}}\lvert^{\sum_{j=1}^{i-1}\alpha_j{-\frac{n-\alpha_{i}}{2}}}\otimes\]\[\otimes \ldots\otimes \lvert\det_{G_{\alpha_k}}\lvert^{\frac{n-\alpha_k}{2}}.\]
\item The cuspidal support of $\lvert\det_{G_n}\lvert^s$ equals to $\lvert-\lvert^{s+\frac{1-n}{2}}+\ldots+\lvert-\lvert^{s+\frac{n-1}{2}}$.
\end{enumerate}
\end{theorem}
If $\pi$ is an irreducible representation of $G_n$, define ${}^\mathfrak{c}\pi$ as the representation obtained by twisting the action of $G_n$ by the generator $\mathfrak{c}\in \mathrm{Gal}(E/F)$, where $\mathfrak{c}$ acts on $g\in G_n$ componentwise.

\subsection{ Classical groups}
Recall now the space $W$ from the introduction, whose Witt-index we denote by $q_W$, \emph{i.e.} the maximal dimension of an isotropic subspace. We recall that if $X_t$ is an isotropic subspace of $W$ with $\dim_\Ee W=t$, we have a decomposition $W=X_t\oplus W_t\oplus Y_t$, where $Y_t$ is another isotropic subspace of $W$ of dimension $t$, orthogonal to $X_t$ and $W_t$ is a $-\epsilon$-herimitan space of dimension $n-2t$. We recall then that the stabilizer of this space is isomorphic to a parabolic subgroup with Levi-component isomorphic to $G_t\times G(W_t)$. We fix for the rest of the paper a minimal parabolic subgroup $B_W$ of $G(W)$ and denote for $\ain{t}{1}{q_W}$ by $Q_t$ the parabolic subgroup containing $B_W$ with Levi-component $G_t\times G(W_t)$. More generally, consider a flag of isotropic subspace $F=(0=X^0\subseteq X^1\subseteq \ldots\subseteq X^{k})$ with $\dim_\Ee X^{i-1}\backslash X^{i}=\alpha_i$. Then the Levi-component of the stabilizer of the flag $Q(F)$ is of the form $G_{\alpha_1}\times\ldots \times G_{\alpha_k}\times G(W_{\alpha_1+\ldots+\alpha_k})$. We denote the standard parabolic subgroup corresponding to such a partition $\alpha$ with $\alpha_1+\ldots+\alpha_k\le q_W$ by $Q_\alpha$.
We then write for \[\pi_1\otimes\ldots\otimes \pi_k\otimes \tau\in \Rep(G_{\alpha_1})\times \ldots \times \Rep(G_{\alpha_k})\times G(W_{\alpha_1+\ldots+\alpha_k})\] for the parabolically induced representation
\[\pi_1\times\ldots\times \pi_k\rtimes\tau\coloneq \id_{Q_\alpha}^{G(W)}(\pi_1\otimes\ldots\otimes \pi_k\otimes \tau).\]
 Recall the M{\oe}glin-Vign\'eras-Waldspurger involution, \emph{cf.} \cite[p.91]{MVW}, \[\mathrm{MVW}\colon \Rep(G(W))\rightarrow \Rep( G(W))\] which is covariant, exact and satisfies
\begin{enumerate}
    \item $\mathrm{MVW}\circ\mathrm{MVW}=\mathrm{id}$,
    \item $\pi\mvw\cong \pi^\lor$ if $\pi$ is irreducible,
    \item For $\ain{t}{1}{q_W}$, $\alpha=(\alpha_1,\ldots,\alpha_k)$ a partition of $t$, $\pi_i\in \irr(G_{\alpha_i})$, $\tau\in\irr(G(W_t))$,\[(\pi_1\times\ldots\times \pi_k\rtimes \tau)\mvw\cong{}^\mathfrak{c}\pi_1\times\ldots\times {}^\mathfrak{c}\pi_k\rtimes\tau\mvw.\]
\end{enumerate}
Let $\ain{t,r}{1}{q_w}$ and
we define \[\mathfrak{B}(t,r)\coloneq\{(k_1,k_2,k_3)\in \NN^3:\, k_1+k_2+k_3=t,\, k_1+k_3\le r,\, r+k_2\le q_W\}.\]
 To $(k_1,k_2,k_3)\in \mathfrak{B}(t,r)$ we can associate an certain element $w_{k_1,k_2,k_3}\in \gl(W)$ with the following properties. Conjugating the Levi-subgroup of the parabolic subgroup $Q_{(k_1,k_2,k_3,r-k_1-k_2)}$ by $w_{k_1,k_2,k_3}$ exchanges the two factors of $G_{k_2}$ and $G_{r-k_1-k_2}$, but leaves them otherwise unchanged, and acts on the two factors $G_{k_1}$ and $G(W_{t-r+k_1+k_3})$ trivially. Finally, $w_{k_1,k_2,k_3}$ preserves $G_{k_3}$. If $\pi_3$ is an irreducible representation of $G_{k_3}$, twisting the action by conjugation by $w_{k_1,k_2,k_3}$ yields the representation ${}^\mathfrak{c}\pi_3^\lor$. We denote by $\mathrm{Ad}(w_{k_1,k_2,k_3})$ conjugation by $w_{k_1,k_2,k_3}$.
\begin{lemma}[{(Geometric Lemma), \cite{Ber}, \cite[Lemma 5.1]{TADIC19951}}]\label{L:GL}
Let $\ain{t,r}{1}{q_w}$ and $\pi$ an admissbile representation of $G_t$ and $\tau$ an admissible representation of $G(W_t)$. 
Then $r_{Q_r}(\pi\rtimes\tau)$ admits a filtration whose subquotients are indexed by $(k_1,k_2,k_3)\in \mathfrak{B}(t,r)$ and are of the following form: \[\Pi_{k_1,k_2,k_3}=\id_{P_{k_1,r-k_1,k_2}\times Q_{k_2}}^{G_r\times G(W_r)}(\mathrm{Ad}(w_{k_1,k_2,k_3})(r_{k_1,k_2,k_3}(\pi)\otimes r_{Q_{r-k_1-k_3}}(\tau))).\]
\end{lemma}
If we only care about the semisimplification of $\Pi_{k_1,k_2,k_3}$, it thus takes the following form.
Write the semisimplification of $r_{(k_1,k_2,k_3)}(\pi)$ as the sum of irreducible representations of the form $\pi_1\otimes \pi_2\otimes \pi_3$ and the semisimplification of $r_{Q_{(r-k_1-k_3)}}(\tau)$ as the sum of irreducible representations of the form $\pi_4\otimes \tau'$. Then 
\[[r_{Q_{r}}(\pi\rtimes \tau)]=\sum [ \pi_1\times \pi_4 \times {}^\mathfrak{c}\pi_3^\lor\otimes \pi_2\rtimes \sigma'],\]
where the sum is over all $k_1,k_2,k_3$ and $\pi_1,\pi_2,\pi_3,\pi_4,\sigma'$ as above. 

Let us use this opportunity to observe the following consequence of the Geometric Lemma of \cite{Ber}: \[r_{Q_t\times G(W)}(S(G(W))\cong \id_{G_t\times G(W_t)\times Q_t}^{G_t\times G(W_t)\times G(W)}(S(G_t\times G(W_t)).\]
\subsection{ Derivatives}
In this section we are following the treatment of \cite[Appendix C.1]{atobe2024localintertwiningrelationscotempered}.
Let $\xi\colon \Ee^\times\ra \C^\times$ be a character such that $\xi\ncong \xi^\lor$. For $\pi\in \irr(G(W))$ we let $r\in \NN$ be maximal such that there exists $\tau\in \irr(G(W_r))$ with $\pi\hra \xi^{\times r}\rtimes \tau$.
\begin{lemma}
    Both $r$ and $\tau$ are, up to isomorphism, uniquely determined by $\pi$. 
\end{lemma}
We thus denote from now on $D_\xi(\pi)\coloneq \tau$ and record the following useful properties of the representation $D_\xi(\pi)$ proven in \cite[Appendix C.1]{atobe2024localintertwiningrelationscotempered}.
\begin{lemma}\label{L:propder}
Let $\pi\in\irr(G(W))$ and $\xi$ a character such that $\xi\ncong\xi^\lor$.
\begin{enumerate}
    \item $\mathrm{soc}(\xi^{\times r}\times D_\xi(\pi))\cong\pi\cong \mathrm{cosoc}((\xi^\lor)^{\times r}\times D_\xi(\pi))$. In particular, if $\pi'\in \irr(G(W))$ is a second representation with $D_\xi(\pi)\cong D_\xi(\pi')$, then $\pi\cong \pi'$.
    \item $D_\xi(\pi)^\lor=D_\xi(\pi^\lor)$.
    \item $\xi^{\times r}\otimes D_\xi(\pi)$ appears with multiplicity one in $r_{Q_r}(\pi)$ and as an irreducible summand.
    \item If for $\tau'\in \irr(G(W_{r'}))$, $r'\ge r$, the representation $\xi^{\times r'}\otimes \tau'$ appears as an irreducible subquotient in $r_{Q_{r'}}(\pi)$, $r=r'$ and $\tau'\cong D_\xi(\pi)$.
\end{enumerate}
\end{lemma}
\subsection{ Degenerate principal series}
We follow the presentation of \cite{KudlaRallis} in this section.
Let $W^-$ be the $-\epsilon$-hermitian space obtained from $W$ by replacing the hermitian form $\langle-,-\rangle$ on $W$ by -$\langle-,-\rangle$. Let $\Delta W\subseteq W+W^-$ be a maximal isotropic subspace of $W+W^-$ and note that $Q(\Delta W)$ has Levi-component $G_n$.
Furthermore, choose $\chi\colon \Ee^\times\ra \C^\times$ a quadratic character and $s\in \C$. 
We then set 
\[I(s,\chi)\coloneq \restr{\id_{Q(\Delta W)}^{G(W+W^-)}(\chi(\det_{G_n})\lvert\det_{G_n}\lvert^s)}{G(W)\times G(W^-)},\]
where we embed $G(W)\times G(W)\cong G(W)\times G(W^-)\hra G(W+W^-)$ via the natural map.
\begin{theorem}[{\cite[Theorem 3]{DroKudRa}}]\label{T:cons}
    Let $\pi\in \irr(G(W))$. Then for all $s\in \C$ and quadratic characters $\chi$,
    \[\dim_\C\ho_{G(W)\times G(W)}(I(s,\chi),\pi\otimes \chi(\det_{G(W^-)})\pi^\lor)=1.\]
\end{theorem}
\begin{prop}[{\cite{KudlaRallis}}]
    The representation $I(s,\chi)$ admits a filtration
    \[0\subseteq I_0\subseteq ...\subseteq I_{q_W}=I(s,\chi)\]
    with \[\sigma_{t,s}\coloneq I_{t-1}\backslash I_t\cong \id_{Q_t\times Q_t}^{G(W)\times G(W^-)}(\chi(\det_{G_t})\lvert\det_{G_t}\lvert^{s+\frac{t}{2}}\otimes \]\[\otimes \chi(\det_{G_t})\lvert\det_{G_t}\lvert^{s+\frac{t}{2}}\otimes (\mathbf{1}_{G(W_t)}\otimes \chi(\det_{W_t^-}))S(G(W_t)))).\]
\end{prop}
\begin{lemma}\label{L:only}
    Assume that $s\in \mathbb{R}_{\ge 0}$. Then  \[\dim_\C\ho_{G(W)\times G(W)}(\sigma_{t,s},\pi\otimes \chi(\det_W)\pi')\neq 0\] only if $\pi'\cong \pi^\lor$.
\end{lemma}
\begin{proof}
We first note that one can easily reduce the claim to the case where $\chi$ is trivial, which we will assume from now on. We will now argue by induction on $\dim_\Ee W$.

Firstly, if $t=0$ recall that $S(G(W))$ is dual to the space of locally constant functions on $G(W)$, denoted by $C^\infty(G(W))$ and therefore
\[\dim_\C\ho_{G(W)\times G(W)}(S(G(W)),\pi\otimes\pi')=\]\[=\dim_\C\ho_{G(W)\times G(W)}(\pi^\lor\otimes\pi'^\lor,C^\infty(G(W)))=\]\[=\dim_\C\ho_{\Delta G(W)}(\pi^\lor\otimes\pi',\C)=\delta_{\pi^\lor,\pi'^\lor}.\] Thus the claim follows easily if $t=0$.

We thus assume from now on that $t>0$ and set $\xi\coloneq \lvert-\lvert^{-s-t+\frac{1}{2}}$, which according to our assumption satisfies $\xi\ncong\xi^\lor$. 
Note that if the Hom-space is non-zero, we can find $\tau\in \irr(G(W_t))$ such that $\lvert\det_{G_t}\lvert^{s+\frac{t}{2}}\rtimes\tau\sra\pi$ and hence by the MVW-involution and \Cref{T:induced},
\[\pi\hra \lvert\det_{G_t}\lvert^{-s-\frac{t}{2}}\rtimes\tau\hra\xi\times \lvert\det_{G_{t-1}}\lvert^{-s-\frac{t-1}{2}}\rtimes \tau.\]

By Frobenius reciprocity and \Cref{L:propder}(4), $D_{\xi}(\pi)\ncong \pi$ and hence we find $r\ge 1$ such that $\pi\hra \xi^{\times r}\rtimes D_\xi(\pi)$. As for $\pi'$, we also have that $D_\xi(\pi')\neq \pi'$, so we obtain an embedding $\pi'\hra\xi^{\times r'}\times D_\xi(\pi')$ with $r'\ge 1$. We will assume thus without loss of generality $r\ge r'$.

Given a non-zero morphism $\sigma_{t,s}\ra\pi\otimes\pi'$, we therefore obtain from Frobenius reciprocity a morphism
\[r_{Q_r\times G(W)}(\sigma_{t,s})\ra\xi^{\times r}\otimes D_\xi(\pi)\otimes \pi'.\]
By the Geometric Lemma and \Cref{T:induced}, $r_{Q_r\times G(W)}(\sigma_{t,s})$ has a filtration with subquotients of the following form:

Let $k_1+k_2+k_3=t$, $k_1+k_3\le r$, $r+k_2\le q_W$. Then \[\Pi_{k_1,k_2,k_3}=\id_{P_{(k_1,r-k_1-k_3,k_3)}\times Q_{k_2}\times Q_{(t,r-k_1-k_3)}}^{G_r\times G(W_r)\times G(W^-)}(\lvert\det_{G_{k_1}}\lvert^{s+\frac{k_1}{2}}\otimes \lvert\det_{G_{k_2}}\lvert^{s+k_1+\frac{k_2}{2}}\otimes  \]\[ \otimes \lvert\det_{G_{k_3}}\lvert^{-s-\frac{t+k_1+k_2}{2}}\otimes\lvert\det_{G_t}\lvert^{s+\frac{t}{2}}\otimes  S(G_{r-k_1-k_3})\otimes S(G(W_{t+r-k_1-k_3})))).\]
Note that there cannot exist a non-zero morphism $\Pi_{k_1,k_2,k_3}\ra\xi^{\times r}\otimes D_\xi(\pi)\otimes \pi'$ unless $k_1=0,\, k_2=t-1,\, k_3=1$. 
Indeed, if $k_1\ge 1$, it would imply that $\lvert-\lvert^{s+\frac{1}{2}}$ would appear in the cuspidal support of $\xi^{\times r}$, thus $\xi=\lvert-\lvert^{s+\frac{1}{2}}$ and hence $2s=-t<0$, a contradiction. If $k_3\ge 2$, it would imply that $\xi\lvert-\lvert$ would appear in the cuspidal support of $\xi^{\times r}$, another contradiction. Finally, if $k_3=0$, it would imply that there exists a suitable representation $\tau'$ with $\xi^\lor\rtimes\tau'\sra D_\xi(\pi)$, since  \[\xi^\lor\times \lvert\det_{G_{t-1}}\lvert^{s+\frac{t-1}{2}}\lvert \sra\lvert\det_{G_t}\lvert^{s+\frac{t}{2}}\] by \Cref{T:induced}, a contradiction to \Cref{L:propder}(4).
Thus \[\ho_{G(W)\times G(W)}(\sigma_{t,s},\pi\otimes \pi')\hra\ho_{G_r\times G(W_r)\times G(W)}(\Pi_{0,t-1,1},\xi^{\times r}\otimes D_\xi(\pi)\otimes \pi').\]
For each map in the latter space, we can apply Bernstein reciprocity to obtain a map
\[\id_{P_{(r-1,1)}\times Q_{t-1}\times Q_{t}\times G_{r-1}}^{G_r\times G(W_r)\times G_{r-1}\times G(W_r)}(\lvert\det_{G_{t-1}}\lvert^{s+\frac{t-1}{2}} \otimes \xi\otimes  \]\[\otimes\lvert\det_{G_t}\lvert^{s+\frac{t}{2}}\otimes  S(G_{r-1})\otimes S(G(W_{t+r-1}))))\sra \xi^{\times r}\otimes D_\xi(\pi)\otimes r_{\overline{Q_{r-1}}}(\pi').\]
We now focus for a moment on the so obtained $G_r\times G_{r-1}$-equivariant morphism \[\xi\times S(G_{r-1})\sra\xi^{\times r}\otimes r_{\overline{Q_{r-1}}}(\pi').\] Denote for $\tau$ an irreducible representation of $G_r$ the up to a scalar unique map $f_\tau\colon S(G_r)\ra \tau\otimes \tau^\lor$.
 Note now that by \cite[Lemma 2.9.3]{DroKudRa}, the above morphism factors through the morphism \[\id_{P_{(1,r-1)}\times G_{r-1}}^{G_r\times G_{r-1}}(\mathds{1}_\xi \otimes f_{\xi^{r-1}})\colon \xi\times S(G_{r-1})\ra\xi^{\times r}\otimes (\xi^{\times r-1})^\lor.\]
It follows that we obtain an embedding
\[\ho_{G(W)\times G(W)}(\sigma_{t,s},\pi\otimes \pi')\hra\ho_{G(W_r)\times G(W)}(\id_{Q_{t-1}\times Q_{t,r-1}}^{G(W_r)\times G(W)}\]\[(\lvert\det_{G_{t-1}}\lvert^{s+\frac{t-1}{2}} \otimes \lvert\det_{G_t}\lvert^{s+\frac{t}{2}}\otimes \xi^{\times r-1}\otimes S(G(W_t)))),D_\xi(\pi)\otimes \pi').\]
By \Cref{T:induced} we have \[\lvert\det_{G_t}\lvert^{s+\frac{t}{2}}\times (\xi^{\times r-1})^\lor\cong (\xi^{\times r-1})^\lor\times \lvert\det_{G_t}\lvert^{s+\frac{t}{2}}\] and the latter is an irreducible quotient of $(\xi^{\times r})^\lor\times \lvert\det_{G_{t-1}}\lvert^{s+\frac{t-1}{2}}$. 
We thus obtain an embedding
\[\ho_{G(W)\times G(W)}(\sigma_{t,s},\pi\otimes \pi')\subseteq\]\[\subseteq \ho_{G(W_r)\times G(W)}(\id_{G(W_r)\times Q_r}^{G(W_r)\times G(W)}((\xi^r)^\lor\otimes \sigma_{t-1,s}),D_\xi(\pi)\otimes \pi'),\]
which by Bernstein-reciprocity is nothing but 
\[\ho_{G(W_{r})\times G_r\times G(W_r)}((\xi^r)^\lor\otimes \sigma_{s,t-1}, D_\xi(\pi)\otimes r_{\overline{Q_{r}}}(\pi')).\]
It follows from \Cref{L:propder} that $r'\ge r$ and since we already assume $r\ge r'$, we have $r=r'$. 
Moreover, by the induction hypothesis the image of any non-zero morphism in the last space must contain $(\xi^r)^\lor\otimes D_\xi(\pi)\otimes D_\xi(\pi)^\lor$, hence \Cref{L:propder} shows that $D_\xi(\pi^\lor)\cong D_\xi(\pi)^\lor\cong D_\xi(\pi')$. By \Cref{L:propder}(1) it thus follows that $\pi^\lor\cong \pi'$.
\end{proof}
\begin{theorem}\label{T:soupedup}
    Let $\pi\in \irr(G(W))$ and assume $s\in \mathbb{R}_{\ge 0}$. Then
    \[\dim_\C\ho_{G(W)\times G(W)}(I(s,\chi),\pi\otimes \chi(\det_{G(W^-)})\pi')=\delta_{\pi',\pi^\lor}.\]
\end{theorem}
\begin{proof}
    If $\pi'\cong \pi^\lor$, the one-dimensionality of the space follows from \Cref{T:cons}. On the other hand, if there exists a non-zero morphism \[I(s,\chi)\ra\pi\otimes \chi(\det_{G(W^-)})\pi',\] there exists $\ain{t}{1}{q_W}$ and a non-zero morphism $\sigma_{t,s}\ra\pi\otimes \chi(\det_{G(W^-)})\pi'$, which by \Cref{L:only} shows that $\pi'\cong \pi^\lor$.
\end{proof}
\begin{rem}\label{Rem}
    \Cref{T:soupedup} does not hold for arbitrary $s\in\C$, as the following example shows. Let $V^+$ and $V^-$ be the two orthogonal spaces with trivial discriminant character and Hasse-Weil invariant $\pm 1$ and dimension $n+1$. Set \[\pi_+=\Theta_{W,V^+,\psi}(\mathbf{1}_{H(V^+)}),\, \pi_-=\Theta_{W,V^-,\psi}(\mathbf{1}_{H(V^-)}).\] Then $\pi_+,\pi_-\in \irr(G(W))$ and they are non-isomorphic. For $s_0=-\frac{n}{4}$, the degenerate principal series $I(s_0,\chi)$ admits as a quotient \[\mathbf{1}_{G_{\frac{n}{2}}}\rtimes \mathbf{1}_1\otimes \mathbf{1}_{G_{\frac{n}{2}}}\rtimes \mathbf{1}_1,\]
    which by \cite[Proposition 7.2]{GanIchino} is of the form $(\pi_+\oplus\pi_-)\otimes(\pi_+\oplus \pi_-)$.
\end{rem}
\section{Howe duality conjecture in type I}
We will now come to the main theorem.
\begin{theorem}
    Let $\pi,\pi'\in \irr(G(W))$. Then:
    \begin{enumerate}
        \item The representation $\theta_{V,W,\psi}(\pi)$ is irreducible.
        \item If $\theta_{V,W,\psi}(\pi)\cong \theta_{V,W,\psi}(\pi')\neq 0$, then $\pi\cong \pi'$.
    \end{enumerate}
\end{theorem}
\begin{proof}
    Let $\chi_W$ be the discriminant character of $W$ and $\chi_V$ the discriminant character of $V$. As explained for example in \cite[§4]{GanTakeda}, the see-saw-diagram
    \[\begin{tikzcd}
G(W+W^-)\arrow[rrd,dash]\arrow[d,dash]&&\arrow[lld,dash]\arrow[d,dash]H(V)\times H(V)\\
        G(W)\times G(W^-)&&\Delta H(V)
    \end{tikzcd}\]
    shows the equality
    \[\dim\ho_{G(W)\times G(W)}(\Theta_{V,W+W^-,\psi}(\chi_W),\pi'\otimes \chi_V(\det_{G(W)})\pi^\lor)=\]\[=\dim_\C\ho_{\Delta H(V)}(\Theta_{V,W,\psi}(\pi')\otimes \Theta_{V,W,\psi}(\pi)^{\mathrm{MVW}},\C)\] and it is easy to see that the dimension of the latter space is bounded below by\[\dim_\C\ho_{H(V)}(\theta_{V,W,\psi}(\pi), \theta_{V,W,\psi}(\pi')).\]
    It thus suffices to show that \[\dim\ho_{G(W)\times G(W)}(\Theta_{V,W+W^-,\psi}(\chi_W),\pi'\otimes \chi_V(\det_{G(W)})\pi^\lor)=\delta_{\pi',\pi}.\]
    The inequality $``\ge"$ is easy to see. To obtain the inequality $``\le "$, we recall from \cite[Proposition 8.2]{GanIchino}
    that \[I(s_{V,W},\chi_v)\sra \restr{\Theta_{V,W+W^-,\psi}(\chi_W)}{G(W)\times G(W^-)},\] where $s_{V,W}=\frac{n+\epsilon_0-m}{2}$ and \[\epsilon_0=\begin{cases}
        \epsilon& E=F,\\0& [E:F]=2.
    \end{cases}\]
Thus \[\dim\ho_{G(W)\times G(W)}(\Theta_{V,W+W^-,\psi}(\chi_W),\pi'\otimes \chi_V(\det_{G(W)})\pi^\lor)\le \]\[\le \dim\ho_{G(W)\times G(W)}(I(s_{V,W},\chi_V)\pi'\otimes \chi_V(\det_{G(W)})\pi^\lor).\]
    If $m\le n+\epsilon_0$, the claim thus follows immediately from \Cref{T:soupedup}.
    If $m>n+\epsilon_0$, we note that $\theta_{V,W,\psi}(\pi)\neq 0$ implies that $\pi$ is a direct summand of $\theta_{W,V,\psi}(\theta_{V,W,\psi}(\pi))$, hence we can just switch the role of $V$ and $W$ and the claim follows precisely by the same argument.
\end{proof}
\bibliographystyle{amsalpha}
\bibliography{reference.bib}
\end{document}